\documentclass{article}
\usepackage[utf8]{inputenc}
\usepackage{amssymb}
\usepackage{amsmath}
\usepackage{amsthm}
\usepackage{setspace}
\usepackage{bm}
\usepackage{graphicx}
\usepackage{color}
\usepackage{hyperref}

\parskip \baselineskip
\newcommand{\Q}[0]{\mathbb{Q}}
\newcommand{\Z}[0]{\mathbb{Z}}
\newcommand{\R}[0]{\mathbb{R}}

\DeclareMathOperator{\conv}{conv}

\DeclareMathOperator{\proj}{proj}
\DeclareMathOperator{\aff}{aff}

\DeclareMathOperator{\st}{s.t.}

\DeclareMathOperator{\depth}{depth}
\DeclareMathOperator{\volume}{volume}

\newtheorem{definition}{Definition}
\newtheorem{proposition}{Proposition}
\newtheorem{theorem}{Theorem}
\newtheorem{corollary}{Corollary}
\newtheorem{lemma}{Lemma}

\newcommand{\floor}[1]{\left\lfloor #1 \right\rfloor}
\newcommand{\ceil}[1]{\left\lceil #1 \right\rceil}
\newcommand{\vvec}[1]{\left[ \begin{array}{c} #1 \end{array} \right]}
\newcommand{\vmat}[2]{\left[ \begin{array}{#1} #2 \end{array} \right]}

\newcommand{\norm}[1]{{\left\lVert#1\right\rVert}}

\begin{document}

\title{On the depth of cutting planes}

\author{Laurent Poirrier, James Yu%
\thanks{Both authors are supported by NSERC Discovery grant RGPIN-2018-04335.}%
\\
\texttt{\small \{lpoirrier,jj6yu\}@uwaterloo.ca}%
}


\maketitle

\begin{abstract}
We introduce a natural notion of depth that applies to individual cutting planes as well as entire families. This depth has nice properties that make it easy to work with theoretically, and we argue that it is a good proxy for the practical strength of cutting planes. In particular, we show that its value lies in a well-defined interval, and we give parametric upper bounds on the depth of two prominent types of cutting planes: split cuts and intersection cuts from a simplex tableau. Interestingly, these parametric bounds can help explain behaviors that have been observed computationally. For polyhedra, the depth of an individual cutting plane can be computed in polynomial time through an LP formulation, and we show that it can be computed in closed form in the case of corner polyhedra.
\end{abstract}

\section{Introduction}
Cutting planes (also called valid inequalities or cuts) have proven computationally useful since the very early history of Integer Programming (IP). In the pioneering work of Dantzig, Fulkerson and Johnson~\cite{DantzigFulkersonJohnson1954} on the traveling salesman problem, cutting planes were the main computational tool alongside the simplex method. They were initially less successful for general IP problems, where they were largely outperformed by Land and Doig's branch-and-bound method~\cite{LandDoig1960}. However, in the mid-nineties, Balas, Ceria and Cornuéjols~\cite{BalasCeriaCornuejols1996} proposed combining the two: first adding a few rounds of cuts, before resorting to branch-and-bound. Since then, cutting planes have become a central component of IP solvers~\cite{Bixby2012}, the most effective single families of cuts being Gomory's Mixed-Integer (GMI) cuts~\cite{Gomory1960}, and Nemhauser and Wolsey's Mixed-Integer Rounding (MIR) cuts~\cite{NemhauserWolsey1988,NemhauserWolsey1990}.

In theory, both GMI and MIR cuts are equivalent~\cite{CornuejolsLi2000} to the split cuts of Cook, Kannan and Schrijver~\cite{CookKannanSchrijver1990}. Let $P$ be a polyhedron in $\R^n$ and let $S$ be the set of its integer points, i.e. $S := P \cap \Z^n$. Consider the integer vector $\pi \in \Z^n$ and the integer $\pi_0 \in \Z$. Clearly, any point $x \in S$ satisfies either $\pi^T x \leq \pi_0$ or $\pi^T x \geq \pi_0 + 1$. A split cut for $P$ is any inequality that is valid for $P^{(\pi, \pi_0)} := \conv(\{ x \in P \; : \; \pi^T x \leq \pi_0 \} \cup \{ x \in P \; : \; \pi^T x \geq \pi_0 + 1 \})$. The split-cut view perfectly illustrates one major issue with cutting planes for general IPs: cut selection. Split cuts are parametrized on $(\pi, \pi_0)$, yielding a huge family of cutting planes. One must choose, among those, a subset that is computationally useful: cuts that yield a reduction in the number of branch-and-bound nodes, without making the formulation too much larger and slowing down the simplex method. This motivates the need for some (possibly heuristic) measure of cut \emph{strength}.

A recent survey by Dey and Molinaro~\cite{DeyMolinaro2018} provides an excellent summary of the current state of the cut selection problem. The survey emphasizes how crucial cut selection is in the implementation of integer programming solvers, but highlights one issue: While the strength of \emph{families} of cuts is well covered in the literature, and many ad-hoc cut evaluation methods exist in practice, there is a dearth of \emph{applicable} theory. In particular, theoretical approaches do not usually allow us to evaluate the strength of individual cuts. As the authors note, ``when it comes to cutting-plane use and selection our scientific understanding is far from complete.'' We only present here a brief introduction to the topic.

The most straightforward notion of strength is based on \emph{closures} and \emph{rank}. We illustrate these concepts with splits, although they apply more generally. The first split closure $P^1$ is the intersection of all split cuts for the formulation $P$, i.e. $P^1 := \bigcap_{(\pi, \pi_0) \in \Z^n \times \Z} P^{(\pi, \pi_0)}$. The process can be iterated, and the $k$th split closure $P^k$ is the first split closure of $P^{k-1}$. A valid inequality for $S$ has rank $k$ if it is valid for $P^k$ but not for $P^{k-1}$. Similarly, the set $\conv(S)$ itself is said to have rank $k$ if $\conv(S) = P^k$ but $\conv(S) \subsetneq P^{k-1}$. From a computational perspective, it is easy to realize how cuts with higher split rank can be seen as stronger. However, this measure is very coarse: One cannot discriminate between, say, cuts of rank one. Furthermore, while the rank is a very powerful theoretical tool, it is impractical: For arbitrary cuts, the split rank is hard to compute exactly\footnote{Given a valid inequality $\alpha^T x \geq \beta$, just determining whether it has rank one is $\mathcal{NP}$-complete already~\cite{CapraraLetchford2003}}, and for split cuts computed by traditional means (successive rounds of cuts), a high rank is paradoxically counter-productive, because numerical errors accumulate with each round.

The other natural candidate for measuring the strength of a valid inequality $\alpha^T x \geq \beta$ is the volume that it cuts off, i.e. $\volume(\{ x \in P : \alpha^T x < \beta\})$. This volume is easy to compute in some important cases (e.g., when $P$ is a cone), but it is more complex to work with in theory. Moreover, comparing volumes presents difficulties. In $\R^n$, volumes can be zero or infinite for entirely legitimate cuts. Zero cases can be mitigated by projecting down to the affine hull of $P$, and infinities by projecting out variables. However, we then obtain incomparable volumes computed in different dimensions. This is not just a theoretical issue. Practical cuts are typically sparse, and if $P$ is a cone, any zero coefficient in the cut will add one unbounded ray to the volume cut off.

On the computational side, a common approach is to evaluate the quantity
\[
\frac{\beta - \alpha^T x^*}{||\alpha||_2}
\]
at a point $x^*$ that we wish to separate, e.g. an LP optimal solution.
This measure has been variously called \emph{(Euclidean) distance}~\cite{WesselmannSuhl2012}, \emph{steepness}~\cite{CookFukasawaGoycoolea2006}, \emph{efficacy}~\cite{FischettiSalvagnin2013} and \emph{depth-of-cut}~\cite{DeyMolinaro2018}. Its limitations are well documented~\cite{DeyMolinaro2018}, but it remains the primary quality measure in academic codes (we can only speculate about what commercial solvers use). For more details, we refer the reader to Wesselmann and Suhl~\cite{WesselmannSuhl2012}, who provide a nice survey of the alternatives (including the \emph{rotated distance} of Cook, Fukasawa and Goycoolea~\cite{CookFukasawaGoycoolea2006}), and computational evaluations.

As an alternative, we propose a measure of strength that we call the \emph{depth} of a cut. Informally, it can be defined as follows. First, the depth of \emph{a point} in $P$ is the distance between this point and the boundary of $P$. Then, the depth of \emph{a cut} is the largest depth of a point that it cuts off. In Section~\ref{sect:depth}, we establish the basic properties of cut depth. In particular, the depth is always positive when the cut is violated, and we show that it is always finite when $S$ is nonempty. Moreover, it is at most $\sqrt{\frac{n+1}{2}}$ when $P$ is full-dimensional, and this bound is tight up to a multiplicative constant. In Section~\ref{sect:bounds}, we apply this concept to give upper bounds on the depth of split cuts. When $P$ is full-dimensional, the upper bound is at most one, emphasizing a large gap between rank-1 split cuts (at most $1$) and the integer hull of $P$ (at most $\sqrt{\frac{n+1}{2}}$). In general, the bound is proportional to the inverse of $||\pi||$. This provides a nice theoretical justification to recent empirical observations that split cuts based on disjunctions with small coefficients are most effective~\cite{FischettiSalvagnin2013,FukasawaPoirrierYang2018}. We then provide a simple bound for intersection cuts from a simplex tableau, which also lines up with some recent computational results. In Section~\ref{sect:algo}, we devise an exact algorithm for computing the depth of an arbitrary cut, which consists in solving a linear programming problem of the same size as the original problem. While this is great in theory, solving one LP per candidate cut could be considered too expensive, still, to warant inclusion into general-purpose IP solvers. However, if we consider corner polyhedra and compute cut depth with respect to their LP relaxation (a simple pointed polyhedral cone), we obtain a much cheaper procedure. We conclude with a few open questions and conjectures in Section~\ref{sect:conclusion}.

\section{Depth}
\label{sect:depth}

We start with a formal definition for the depth of a point, which relies on a uniform notation for balls.
Note that we use $|| \cdot ||_k$ to denote the $L^k$ norm, and that $|| \cdot ||$ indicates the $L^2$ norm when $k$ is omitted.

\begin{definition} We define $B^k(x, r)$ as the ball of radius $r \in \R_+$ in norm $L^k$ centered in $x \in \R^n$, i.e. $B^k(x, r) := \{ y \in \R^n : || y - x ||_k \leq r \}$. When $k$ is omitted, $B(x, r)$ implicitly uses the $L^2$-norm.
\label{def:ball}
\end{definition}
\begin{definition} Let $P \in \R^n$ and $x \in P$.
We define the depth of $x$ with respect to $P$ as
\[
\depth_P(x) := \sup \{ r \in \R_+ : B(x, r) \cap \aff(P) \subseteq P \}.
\]
\label{def:point-depth}
\end{definition}
Observe that, in Definition~\ref{def:point-depth}, we intersect $B(x, r)$ with the affine hull of $P$ in order to properly handle the case in which $P$ is not full-dimensional. Otherwise, the depth would always be zero whenever $\dim(P) < n$. We can now introduce the depth of any subset $S$ of $P$.
\begin{definition} Let $P, S \in \R^n$ be such that $S \subseteq P$.
We define the depth of $S$ with respect to $P$ as
\[
\depth_P(S) := \sup \{ \depth_P(x) : x \in P \setminus S \}.
\]
By extension, we define the depth of an inequality $\alpha^T x \geq \beta$ that is valid for $P \cap \Z^n$ as
\[
\depth_P(\alpha, \beta) := \depth_P(\{ x \in P : \alpha^T x \geq \beta \}).
\]
\label{def:set-depth}
\end{definition}
This bears resemblance with the notion of distance introduced by Dey, Molinaro and Wang~\cite{DeyMolinaroWang2015}, $d(S, P) := \max_{x \in P} \min_{y \in S} || x - y ||_2$. The measures are distinct, but it is easy to show that $\depth_P(S) \leq d(S, P)$ for all $S \subseteq P \subseteq \R^n$. We now show that in the full-dimensional case, 
the depth of a cutting plane is a lower bound on the volume that it cuts off.

\begin{proposition}
Let $P \in \R^n$ be a full-dimensional convex set, let $\alpha^T x \geq \beta$ be
an inequality, and let $V_{\text{cut}}(\alpha, \beta)$ be the volume
of $\{ x \in P : \alpha^T x < \beta \}$. Then,
\[
V_{\text{cut}}(\alpha, \beta) \geq
	\frac{1}{2} V_n(\depth_P(\{ x \in P : \alpha^T x \geq \beta \}))
\]
where $V_n(r)$ is the volume of a ball of radius $r$ in $\R^n$.
Note that, for example,
$V_n(r) = \frac{\pi^{n/2}}{(n/2)!} r^n$ when $n$ is even.
\label{prop:volume}
\end{proposition}
\begin{proof}
Let $S := \{ x \in P : \alpha^T x \geq \beta \}$,
$x' \in P \setminus S$, and $r > 0$ be such that
$B^2(x', r) \subseteq P$. We build the half-ball
$H := \{ x \in B^2(x', r) \; : \; \alpha^T x \leq \alpha^T x' \}$.
Since $\alpha^T x' < \beta$,
we have that $H \subseteq \{ x \in B^2(x', r) \; : \; \alpha^T x < \beta \}
\subseteq \{ x \in P \; : \; \alpha^T x < \beta \}$.
In other words, $H$ is included in the set that is cut off
from $P$ by $\alpha^T x \geq \beta$.
\end{proof}

Clearly, an upper bound on such depth for any valid inequality is the
depth of the integer hull of $P$, i.e. $\depth_P(\conv(P \cap \Z^n))$.
We now give a first upper bound on the latter.

\begin{proposition} Let $P \subseteq \R^n$ be a full-dimensional convex set.
The depth of the integer hull of $P$ is at most
$\sqrt{n}$.
\label{prop:depth-ub}
\end{proposition}
\begin{proof}
Having $\depth_P(\conv(P \cap \Z^n)) > \sqrt{n}$ would imply that there exist
$x \in P \setminus \conv(P \cap \Z^n)$ such that $B^2(x, r) \subseteq P$
with $r > \sqrt{n}$.
We then have $B^\infty(x, 1) \subseteq B^2(x, r) \subseteq P$, which means
that all $2^k$ integer roundings of $x$ are in $P$, hence
contradicting $x \notin \conv(P \cap \Z^n)$.
\end{proof}

We mention the $\sqrt{n}$ upper bound from Proposition~\ref{prop:depth-ub} because it is intuitive and easy to obtain, but this result can be strengthened, and we do so in Theorem~\ref{thm:depth-ub}. In order to prove it, we first establish the following lemma.

\begin{lemma} Let $n \geq 2$. For any point $y \in \R^n$,
there exists a lattice polytope
$X \subseteq \R^n$ such that
$y \in X \subset B\left(y, \sqrt{\frac{n+1}{2}}\right)$.
\label{lemma:X}
\end{lemma}
\begin{proof}
As we can translate the problem by arbitrary integers, we assume without loss of generality that $0 \leq y \leq 1$. Furthermore, by using reflections (i.e., replacing $y_j$ by $(1 - y_j)$ for the appropriate indices $j$), we assume wlog that $y \in Y$, with
$Y := \{ y \in \R^n \, : \, 0 \leq y \leq \frac{1}{2} \}$. Consider the set
$X := \{ x \in \R^n \, : \, 0 \leq x \leq 1 \text{ and } 1^T x \leq \ceil{1^T y} \}$.
Clearly, $y \in X$. Moreover, the matrix defining $X$ is totally unimodular, so all its vertices are integral, and $X$ can be equivalently rewritten
$X = \conv\{ x \in \{ 0, 1 \}^n \, : \, 1^T x \leq \ceil{1^T y} \}$.
We now establish an upper bound on the Euclidean distance between $y$ and any point of $X$. To that end, we look for a pair $\bar y \in Y$, $\bar x \in X$ that maximizes that distance $||\bar y- \bar x||$. Observe that because $X$ and $Y$ are polyhedra, there exists at least one maximizer in which
$\bar x$ is a vertex of $X$ and $\bar y$ is a vertex of $Y$. We can thus write that $||\bar y - \bar x||^2$ is given by
\begin{equation}
\begin{array}{rll}
\max & \sum_j (y_j - x_j)^2 & \\
\st  & y_j \in \{ 0, \frac{1}{2} \} & \text{ for all } j \\
 & x_j \in \{ 0, 1 \} & \text{ for all } j \\
 & 1^T x \leq \ceil{1^T y}. &
\end{array}
\label{eq:depth-ub-opt}
\end{equation}
Observe that the terms $(y_j - x_j)^2$, in the objective function, can only take one of the discrete values $\{ 0, \frac{1}{4}, 1 \}$. We say that term $j$ is a $v$-term if $(y_j - x_j)^2 = v$. For $1$-terms, we have $y_j = 0$ and $x_j = 1$. For $0$-terms, we have $y_j = x_j = 0$.
For every $\frac{1}{4}$-term, we have $y_j = \frac{1}{2}$ and we know that $x_j = 0$ in some optimal solution, since using that value does not affect the $1^T x \leq \ceil{1^T y}$ constraint. Given that restriction, a pair $(x, y)$ is feasible if and only if it satisfies the following condition: if we have $k \geq 1$ $1$-terms, then we need at least $(2k - 1)$ $\frac{1}{4}$-terms in order to satisfy $\sum_j x_j \leq \ceil{\sum_j y_j}$. We can thus construct an optimal solution by greedily maximizing the number of $1$-terms, then the number of $\frac{1}{4}$-terms. We have $n$ terms in total, so
\begin{eqnarray*}
k + (2k - 1) & \leq & n \\
k & \leq & \frac{n+1}{3} \\
k & = & \floor{\frac{n+1}{3}}.
\end{eqnarray*}
We obtain
\begin{equation}
\left\{ \begin{array}{lll}
	\bar y_j = 0 \text{ and } \bar x_j = 1,
		& \text{for } j \in J^1,
		& \text{where } |J^1| = \floor{\frac{n+1}{3}} \\
	\bar y_j = \frac{1}{2} \text{ and } \bar x_j = 0,
		& \text{for } j \in J^{\frac{1}{4}},
		& \text{where } |J^{\frac{1}{4}}| = n - \floor{\frac{n+1}{3}}.
\end{array} \right.
\label{eq:depth-ub-cons}
\end{equation}
It follows that
\[
|| \bar y - \bar x ||^2
= \floor{\frac{n+1}{3}} + \frac{1}{4} \left( n - \floor{\frac{n+1}{3}} \right)
= \floor{\frac{n+1}{3}} + \frac{n}{4} - \frac{1}{4} \floor{\frac{n+1}{3}}.
\]
Let $\rho \in \{ 0, 1, 2 \}$ be such that
$\rho = (n + 1) \pmod{3}$. We get
\begin{eqnarray*}
|| \bar y - \bar x ||^2
 & = & \frac{n+1-\rho}{3} + \frac{n}{4} - \frac{1}{4} \cdot \frac{n+1-\rho}{3} \\
 & = & \frac{1}{12} \left( 4n+4-4\rho + 3n - n+1-\rho \right) \\
 & = & \frac{n}{2} + \frac{5 - 5 \rho}{12},
\end{eqnarray*}
so the maximum distance is
\[
|| \bar y - \bar x || = \sqrt{\frac{n}{2} + \frac{5 - 5 \rho}{12}}
< \sqrt{\frac{n+1}{2}}.
\]
Therefore, for any $y \in \R^n$, we can construct a set $X$ such that
$X$ is a lattice polytope, $y \in X$,
and $X \subset B\left(y, \sqrt{\frac{n+1}{2}}\right)$.
\end{proof}

\begin{theorem} Let $n \geq 2$
and let $P \subseteq \R^n$ be a full-dimensional convex set.
The depth of the integer hull of $P$ is less than $\sqrt{\frac{n+1}{2}}$.
\label{thm:depth-ub}
\end{theorem}
\begin{proof}
Suppose, by contradiction, that
$\depth_P(\conv(P \cap \Z^n)) \geq \sqrt{\frac{n+1}{2}}$.
Then, there exists $y \in P \setminus \conv(P \cap \Z^n)$ such that
$\depth_P(y) \geq \sqrt{\frac{n+1}{2}}$,
i.e., $B \left(y, \sqrt{\frac{n+1}{2}} \right) \subseteq P$.
By Lemma~\ref{lemma:X}, we can construct a polytope $X$ with integral vertices
that satisfies
$y \in X \subset B \left(y, \sqrt{\frac{n+1}{2}} \right) \subseteq P$.
This implies that $y$ is a convex combination of integral points in $P$,
contradicting $y \notin \conv(P \cap \Z^n)$.
\end{proof}

\begin{corollary}
Let $n \geq 2$ and let $P \subseteq \R^n$ be a rational polyhedron
whose integer hull is nonempty.
We assume wlog that $0 \in P$ and
consider a basis $L \in \Z^{n \times d}$ of the lattice
$\Z^n \cap \aff(P)$.
The depth of the integer hull of $P$ is less than
$\sqrt{\lambda_d} \sqrt{\frac{d+1}{2}}$,
where $\lambda_d$ is the largest eigenvalue of $L^TL$.
\end{corollary}
\begin{proof}
Since $P$ is rational, $\aff(P) = \{ Ly \, : \, y \in \R^d \}$,
so every point $\tilde y \in P$ can be expressed as $\tilde y = L y$
for some $y \in \R^d$.
By Lemma~\ref{lemma:X}, for every $y \in \R^d$, there exists
a lattice polytope $X \in \R^d$ such that
$y \in X \subset B\left(y, \sqrt{\frac{d+1}{2}}\right)$.
Observe that $\tilde y \in LX$
and that $LX \subset \R^n$ is also a lattice polytope.
Letting $\tilde x$ be a vertex of $LX$, we know that $\tilde x = L x$ for
some $x \in X$. We can now bound the distance between $\tilde x$ and $\tilde y$.
\begin{eqnarray*}
||\tilde y - \tilde x||^2
 & = & (\tilde y - \tilde x)^T (\tilde y - \tilde x) \\
 & = & (L y - L x)^T (L y - L x) \\
 & = & (L (y - x))^T (L (y - x)) \\
 & = & (y - x)^T L^T L (y - x) \\
 & \leq & \lambda_d ||y - x||^2 \; \leq \; \lambda_d \sqrt{\frac{d+1}{2}}.
\end{eqnarray*}
We can now apply to $\tilde y$ the same reasoning as in the proof of
Theorem~\ref{thm:depth-ub}.
\end{proof}

Whereas Theorem~\ref{thm:depth-ub} only provides an upper bound on the depth of integer hulls, we now construct a lower bound. Theorem~\ref{thm:depth-lb} shows that one can construct formulations whose integer hull has high depth.

\begin{theorem} For $n \geq 2$, there exists a full-dimensional
convex set $P \subseteq \R^n$ whose integer hull has depth
$\depth_P(\conv(P \cap \Z^n)) = \frac{\sqrt{3 + n}}{2}$.
\label{thm:depth-lb}
\end{theorem}
\begin{proof}
We define $T := \{ 0, 1 \}^{n-1}$ and let $T = \{ t^1, \ldots, t^{2^{n-1}} \}$.
We then define $D := \{ d^1, \ldots, d^{2^{n-1}} \} \subseteq \R^n$ where
$d^j = (1, t^j - \frac{1}{2} \bm{1})$.
Observe that $|d^1| = \cdots = |d^{2^{n-1}}|$.
We build the translated polyhedral cone $P$
\[
P := \{ x \in \R^n \; : \; {d^j}^T x \leq 1 + \frac{1}{2} |t^j|_1 - \varepsilon,
	\; \text{ for } j = 1, \ldots, 2^{n-1} \}
\]
where $|t^j|_1$ is the $L^1$ norm of $t^j$, or in this case the number of
ones in the binary vector $t^j$. The parameter $\varepsilon > 0$ can be
arbitrarily small. Finally, we give a point
$c := \left( 1, \frac{1}{2}, \ldots, \frac{1}{2}\right) \in \R^n$.
As illustrated in Figure~\ref{fig:depth-lb},
$P$ is constructed such that the $j$th inequality
is orthogonal to $d^j$ and almost touches but cuts off
$(1, t^j) = c + d^j$. All integer points in $x \in P \cap \Z^n$
satisfy $x_1 \leq 0$, so the point $x' = c + \varepsilon e_1$ does not
belong to $\conv(P \cap \Z^n)$. Therefore, the ball
$B^2(x', |d^1| - \varepsilon')$ is included in $P$ for some $\varepsilon' > 0$.
Note that $\varepsilon'$ can be chosen arbitrarily small as $\varepsilon$
tends towards zero.
At the limit, the depth of the integer hull of $P$ is
$|d^1| = \sqrt{1^2 + (n - 1) \left(\frac{1}{2}\right)^2}
	= \sqrt{\frac{3}{4} + \frac{n}{4}}
	= \frac{\sqrt{3 + n}}{2}$.
\begin{figure}
\centering
\includegraphics[height=7cm]{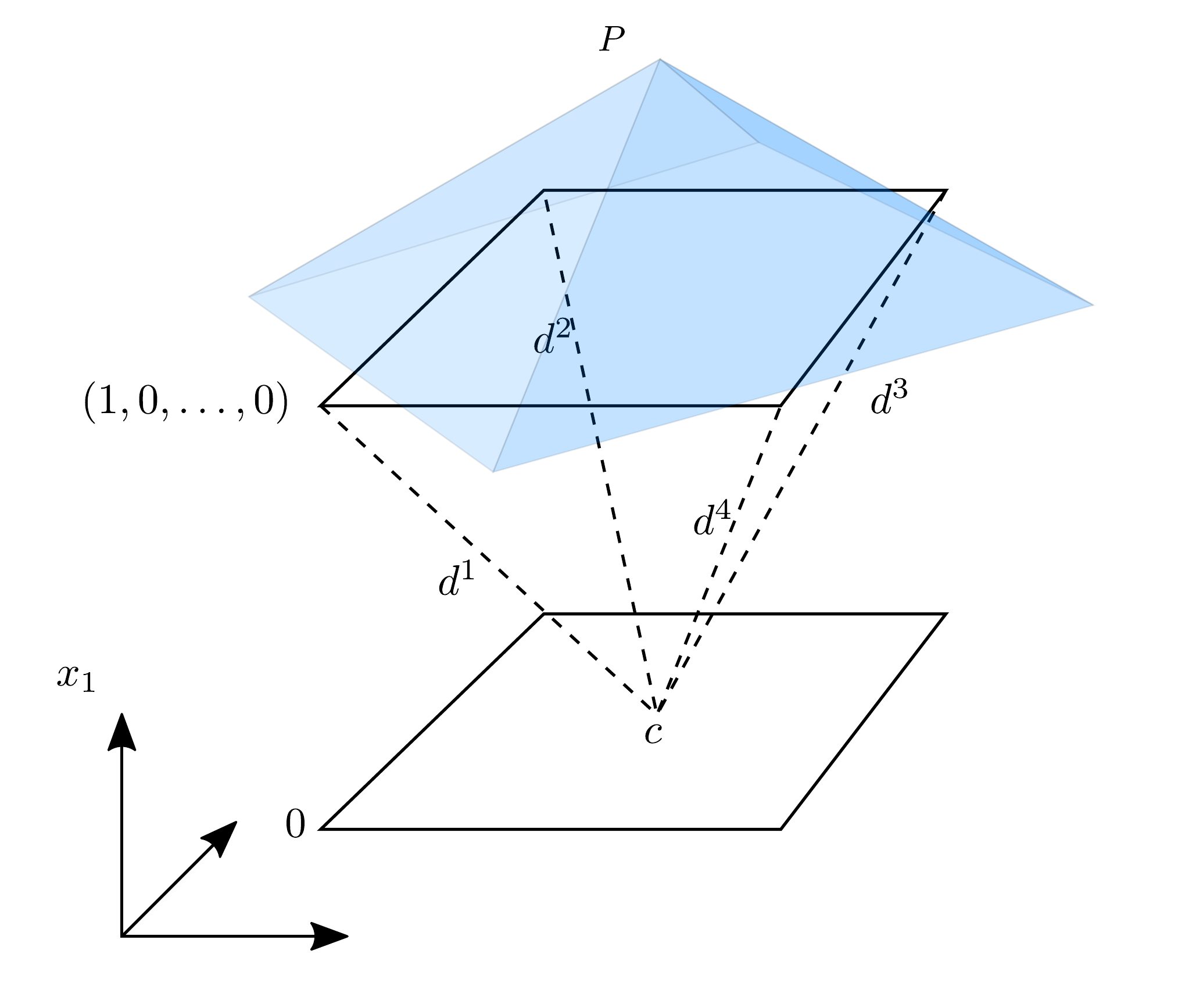} \\
\caption{The cone $P$ in the proof of Theorem~\ref{thm:depth-lb}.}
\label{fig:depth-lb}
\end{figure}
\end{proof}

\section{Bounds on the depth of cutting planes}
\label{sect:bounds}

\subsection{Split cuts}
\label{sect:split-bound}

It has long been understood that split cuts computed from a disjunction $(\pi, \pi_0)$ tend to be more effective computationally when the coefficients of $\pi$ are small. For example, this can be seen in the row aggregation heuristics of Fischetti and Salvagnin~\cite{FischettiSalvagnin2013}. Recently, Fukasawa, Poirrier and Yang~\cite{FukasawaPoirrierYang2018} performed computational tests and explicitly verified this hypothesis.
The intuition behind this behavior is that the distance between the two half-spaces $\pi^T x \leq \pi$ and $\pi^T x \geq \pi + 1$ decreases when $\pi$ grows. Our notion of depth lets us formalize this intuition.

\begin{theorem}\label{thm:cutboundary}
  Let $V$ be the affine hull of $P$ and let $\proj_{V}{y}$ denote the projection of $y \in \R^n$ onto $V$.
  Suppose $x \in P \setminus P^{(\pi, \pi_0)}$. If $\proj_V \pi = 0$, then $P^{(\pi, \pi_0)} = \emptyset$ and $P \cap \Z^n = \emptyset$.
  Otherwise, there exists a point $x' \in V \setminus P$ such that $\norm{x - x'} \le \frac{1}{\norm{\proj_V\pi}}$.
\end{theorem}
\begin{proof}
  Observe that since $x \notin P^{(\pi, \pi_0)}$, we know that $\pi_0 = \floor{\pi^T x}$.
  We construct a vector $d \in V$ given by $d = \frac{1}{{\norm{\proj_V\pi}}^2}\proj_V\pi$.
  Then, we cannot have that both $x + d$ and $x - d$ are in $P$. Indeed,
    \begin{eqnarray*}
      \pi^T\left(x + d\right)
        & = & \pi^T x + \pi^T d \\
        & = & \pi^T x + \frac{\pi^T \proj_V\pi}{{\norm{\proj_V\pi}}^2} \\
        & = & \pi^T x + \frac{\left(\proj_V\pi + \proj_{V^\perp}\pi\right)^T \proj_V\pi}{{\norm{\proj_V\pi}}^2} \\
        & = & \pi^T x + \frac{\left(\proj_V\pi\right)^T\proj_V\pi}{{\norm{\proj_V\pi}}^2} \\
        & = & \pi^T x + 1 \\
        & \ge & \floor{\pi^T x} + 1 = \pi_0 + 1
    \end{eqnarray*}
  and similarly $\pi^T\left(x - d\right) \le \pi_0$.
  Thus if both $x + d$ and $x - d$ were in $P$, we would also have that $x-d, x+d \in P^{(\pi, \pi_0)}$,
  contradicting either $x \notin P^{(\pi, \pi_0)}$ or the convexity of $P^{(\pi, \pi_0)}$.
  Finally, we let $x'$ in the statement be one of $x-d$ or $x+d$ in order to satisfy $x' \notin P$,
  and verify that $||x - x'|| = ||d|| = \frac1{\norm{\proj_V\pi}}$.
\end{proof}
If we do not require $x'$ to lie in the affine hull of $P$, then Theorem~\ref{thm:cutboundary} simplifies into Corollary~\ref{cor:cutboundary}.
\begin{corollary}\label{cor:cutboundary}
  Let $P \subseteq \R^n$.
  If $x \in P \setminus P^{(\pi, \pi_0)}$, then there exists a point $x' \notin P$ such that $\norm{x - x'} \le \frac{1}{\norm{\pi}}$.
\end{corollary}

The bounds in Theorem~\ref{thm:cutboundary} and Corollary~\ref{cor:cutboundary} can be interpreted as follows. For any disjunction $(\pi, \pi_0)$, all the points separated by split cuts from $(\pi, \pi_0)$ are within a distance of $\frac{1}{\norm \pi}$ from the boundary, and a distance of $\frac{1}{\norm{\proj_V\pi}}$ from the boundary within the affine hull of $P$.

\begin{corollary} Given $P \subseteq \R^n$,
\[
\depth_P\left(P^{(\pi, \pi_0)}\right) \leq \frac{1}{\norm{\proj_V\pi}}.
\]
Furthermore, if $P$ is full-dimensional, then
$\depth_P\left(P^{(\pi, \pi_0)}\right) \leq \frac{1}{\norm{\pi}}$.
Given $x \in P \setminus P^{(\pi, \pi_0)}$, we have
$
\depth_P(x) \leq
\frac{\max\{
	\pi^T x - \floor{\pi^T x},
	\ceil{\pi^T x} - \pi^T x \}}{\norm{\proj_V\pi}}
$ in the general case, and
$
\depth_P(x) \leq
\frac{\max\{
	\pi^T x - \floor{\pi^T x},
	\ceil{\pi^T x} - \pi^T x \}}{\norm \pi}
$
if $P$ is full-dimensional.
\label{cor:split-depth}
\end{corollary}

This confirms that cuts with small $||\pi||$ are potentially deeper.
Moreover, Corollary~\ref{cor:split-depth} suggests that disjunction
directions $\pi$ that are close to orthogonal to the affine hull
of $P$ may be more beneficial.
We now show that the above bounds are tight.

\begin{proposition}
The bounds in Theorem~\ref{thm:cutboundary} and Corollary~\ref{cor:cutboundary} are tight.
\end{proposition}
\begin{proof}
  Given any $0 < \varepsilon < 1$, consider the set $P = \{ x\in \R^{n} : x_1 \ge \frac{\varepsilon}{3} \}$,
  the disjunction $(e_1, 0) \in \Z^{n+1}$, and the point $x = \left(1-\frac{\varepsilon}3, 0\right) \in P$.
  It is easy to verify that $x \notin P^{(e_1, 0)}$.
  The affine hull of $P$ is $\R^n$ so Theorem~\ref{thm:cutboundary}
  and Corollary \ref{cor:cutboundary} describe the same bound $\frac{1}{\norm{e_1}} = 1$.
  However, the distance from $x$ to the boundary of $P$ is
    $$1-\frac{2\varepsilon}{3} > 1-\varepsilon = \frac{1}{\norm{e_1}}-\varepsilon.$$
  Letting $\varepsilon \rightarrow 0^+$ yields the claim.
\end{proof}

Note that our results automatically apply to Chvátal-Gomory cuts~\cite{Gomory1958,Gomory1963,Chvatal1973} as well, since Chvátal-Gomory cuts are exactly the split cuts for which one side of the disjunction is infeasible~\cite{BonamiCornuejolsDashFischettiLodi2008}.

\subsection{Intersection cuts}
\label{sect:ic-bound}

Another upper bound we can derive concerns the depth of intersection cuts~\cite{Balas1971} for continuous relaxations of corner polyhedra
(this is the most common setup for the practical generation of so-called
multirow cuts).
Consider a feasible continuous corner
\[
P_I = \left\{ (x, s) \in \Z^m \times \R^n_+ : x = f + R s \right\},
\]
where $f \in \Q^m \setminus \Z^m$ and $R \in \Q^{m \times n}$.
The inequalities $s_j \geq 0$, and any conic combination of these inequalities, are trivially valid for $\conv(P_I)$.
Because $(R^j, e_j)$ is an extreme ray of $\conv(P_I)$ for every $j$,
any nontrivial valid inequality for $\conv(P_I)$ takes the form
$\alpha^T s \geq 1$ where $\alpha \geq 0$. Given any such nontrivial
valid inequality, we can immediately compute an upper bound on its depth,
by inspection.

\begin{theorem}
Let $\alpha^T s \geq 1$ be a valid inequality for $\conv(P_I)$.
Its depth satisfies
\[
\depth_P(\alpha, 1)
	\leq
	\min_{j}
	\left\{
	\frac{1}{\alpha_j} \sqrt{\left( \sum_i R_{ij}^2 \right) + 1}
	\; : \; \alpha_j > 0
	\right\}.
\]
\label{thm:ic-ub}
\end{theorem}
\begin{proof}
Let $(x, s)$ be any point that belongs to the LP relaxation of $P_I$,
i.e. $s \geq 0$, but that is cut off by $\alpha^T s \geq 1$.
For any index $j$ such that $\alpha_j > 0$, we have
$\alpha_j s_j \leq \alpha^T s < 1$, so $s_j < \frac{1}{\alpha_j}$.
We construct a point $(\bar x, \bar s)$ by setting
$\bar s = s - \frac{1}{\alpha_j} e_j$ and $\bar x = f + R \bar s$.
Since $\bar s_j < 0$, $(\bar x, \bar s)$ is LP infeasible.
The depth of the cut is thus upper bounded by the Euclidean distance $d$
between $(x, s)$ and $(\bar x, \bar s)$.
We have
\begin{eqnarray*}
d^2
 & = & \norm{ \vvec{ \bar x \\ \bar s } - \vvec{ x \\ s } }^2 \\
 & = & \norm{ \vvec{ R (\bar s - s) \\ \bar s - s } }^2 \\
 & = & \norm{ \vvec{ R \frac{1}{\alpha_j} e_j \\ \frac{1}{\alpha_j} e_j } }^2 \\
 & = & \frac{1}{\alpha_j^2} \norm{ \vvec{ R e_j \\ e_j } }^2. \\
\end{eqnarray*}
Therefore,
\[
\depth_P(\alpha, 1) \leq
	\frac{1}{\alpha_j}
	\sqrt{\left( \sum_i R_{ij}^2 \right) + 1}.
\]
\end{proof}

There are two interesting links between Theorem~\ref{thm:ic-ub} and the computational literature. First, note that the multipliers $\sqrt{\left( \sum_i R_{ij}^2 \right) + 1}$ coincide with the edge lengths in the steepest edge pricing~\cite{GoldfarbReid1977} for the primal simplex method\footnote{Steepest edge is by far the best practical pricing method to minimize the number of simplex iterations. For the dual simplex method, steepest edge pricing is also the fastest approach overall~\cite{ForrestGoldfarb1992}. In the primal case, it is balanced by higher per-iteration cost, and is often outperformed by faster, approximate versions of steepest edge.}. This is not surprising, since the edge length is the linear coefficient in the relationship between the increase of the value of one nonbasic variable and the Euclidean distance traveled. Still, it is noteworthy that those same edge lengths appear in cut depth computations as well. Second, setting aside the multipliers, we can observe that the tightest upper bound in the proof of Theorem~\ref{thm:ic-ub} is given by the largest value of $\alpha_j$. It thus seems natural that we should minimize the largest $\alpha_j$ coefficient in order to obtain a deep cut. This is exactly the \emph{infinity cut} approach of Fukasawa, Poirrier and Xavier~\cite{FukasawaPoirrierXavier2019}, who showed computationally that selecting a few cuts that minimize the largest $\alpha_j$ can yield about half the performance of selecting \emph{all} the cuts together. Theorem~\ref{thm:ic-ub} suggests that one could possibly improve on~\cite{FukasawaPoirrierXavier2019} by weighing the infinity norm measure using the primal steepest edge lengths.

\section{Computing the depth of a cutting plane}
\label{sect:algo}

\subsection{General polyhedra}

Consider a nonempty polyhedron $P$ described by
\[
P = \{ x \in \mathcal{L} \; : \; A x \leq b \},
\]
where $A \in \R^{m \times n}$, $\mathcal{L} \subseteq \R^n$ is an affine space,
and  $\dim(P) = \dim(\mathcal{L})$,
implying that $\mathcal{L}$ is the affine hull of $P$.
We denote by $a^i$ the $i$th row of $A$. We first characterize the depth of
a point.

\begin{proposition}
A point $\tilde x \in P$ has $\depth_P(\tilde x) \geq \delta$
if and only if, for every $i$, the distance between $\tilde x$
and the affine space $\{ x \in \mathcal{L} : {a^i}^T x = b_i \}$
is at least $\delta$.
\label{prop:point-depth}
\end{proposition}
\begin{proof}
($\Leftarrow$): If the aforementioned distance is at least $\delta$, then
$(B(\tilde x, \delta) \cap \mathcal{L})
	\subseteq \{ x \in \mathcal{L} : {a^i}^T x \leq b_i \}$.
If this is true for all $i$,
we have $(B(\tilde x, \delta) \cap \mathcal{L}) \subseteq P$, hence
$\depth_P(\tilde x) \geq \delta$. ($\Rightarrow$):
Conversely, if the distance is lower than $\delta$ for some $i$, then
there exists a point $z \in (B(\tilde x, \delta) \cap \mathcal{L})$ such that
${a^i}^T x > b_i$. Therefore, $z \notin P$, so $\depth_P(\tilde x) < \delta$.
\end{proof}

Let us denote by $t^i \in \R^n$ the vector that satisfies
(i) $||t^i|| = 1$,
(ii) $t^i \in L$, and
(iii) $\{ x \in \mathcal{L} : {a^i}^T x \leq b_i \}
	= \{ x \in \mathcal{L} : {t^i}^T x \leq h_i \}$
for some $h_i$.
Observe that $t^i$ is orthogonal to
$\{ x \in \mathcal{L} : {a^i}^T x = b_i \}$, and
can be obtained by projecting $a^i$ on $\mathcal{L}$, then normalizing.
Letting $T \in \R^{m \times n}$ be the matrix whose $i$th row is $t^i$
for all $i$, we can rewrite $P$ as
$P = \{ x \in \mathcal{L} \; : \; T x \leq h \}$.

Given $t^i$ and $h_i$, the distance between $\tilde x \in P$ and
$\{ x \in \mathcal{L} : {a^i}^T x = b_i \}$ can be expressed as
\begin{eqnarray}
\max \{ \lambda \geq 0 \; : \; {t^i}^T (\tilde x + \lambda t^i) \leq h_i \}
& = & \max \{ \lambda \geq 0 \; : \; {t^i}^T \tilde x + \lambda \leq h_i \}
	\label{eq:x-dist-max} \\
& = & h_i - {t^i}^T \tilde x.
	\label{eq:x-dist-cf}
\end{eqnarray}
By Proposition~\ref{prop:point-depth}, $\depth_P(\tilde x)$ is
the minimum such distance, yielding
\begin{eqnarray}
\depth_P(\tilde x)
& = & \max \{ \lambda \geq 0 \; : \; T \tilde x - 1.\lambda \leq h \}
	\label{eq:x-depth-max} \\
& = & \min_i \{ h_i - {t^i}^T \tilde x \}.
	\label{eq:x-depth-cf}
\end{eqnarray}
It follows from~\eqref{eq:x-depth-max} that for $\lambda \geq 0$, the set
$P(\lambda) := \{ x \in \mathcal{L} \, : \, T x + 1.\lambda \leq h \}$
is the set of all points of $P$ that have depth at most $\lambda$.

We now tackle the depth of a cut.
Given a cutting plane $\alpha^T x \geq \beta$, the closure of the set of points that are cut off from $P$ by that cutting plane is given by the polyhedron $\{ x \in P \, : \, \alpha^T x \leq \beta \}$. The depth of that cutting plane is the largest $\lambda \geq 0$ such that the set
\begin{equation}
\{ x \in P(\lambda) \, : \, \alpha^T x \leq \beta \}
\label{eq:depth-set}
\end{equation}
is feasible.
We thus find $\depth_P(\alpha, \beta)$ by solving the linear optimization problem
\begin{equation}
\depth_P(\alpha, \beta) =
\begin{array}[t]{rl@{\;}l}
\max & \lambda \\
\st  & T x + 1 . \lambda & \leq h \\
     & \alpha^T x & \leq \beta \\
     & \multicolumn{2}{l}{x \in \mathcal{L}, \lambda \geq 0.}
\end{array}
\label{eq:depth-lp}
\end{equation}

\subsection{Corner polyhedra}

Consider a corner polyhedon~\cite{GomoryJohnson1972I,GomoryJohnson1972II},
for example $\{ (x, s) \in \Z^m \times \Z^n_+ \, : \, x = f + R s \}$,
and let $P$ be its LP relaxation. The polyhedron $P$ can be obtained
by applying an affine transformation to an orthant (that is not
necessarily full-dimensional). Therefore, $P$
is a simple pointed cone, i.e., it has a vertex $v$ and can be expressed
as $v$ plus a conic combination of its $\dim(P)$ extreme rays.
Again, we let $P = \{ x \in \mathcal{L} \, : \, A x \leq b \}$,
where $\mathcal{L}$ is the affine hull of $P$.
Furthermore, we assume wlog that $A$ contains no redundant
inequalities. Then, the vertex of $P$ is
$v := \{ x \in \mathcal{L} \, : \, A x = b \}$.
Let $L \in \R^{(n - \dim(\mathcal{L})) \times n}$ be a full row rank matrix
such that $\mathcal{L} = \{ x \in \R^n \, : \, L x = \xi \}$
for some $\xi \in \R^{(n - \dim(\mathcal{L}))}$.
Using the matrix $T$ defined above, the vertex of $P$ can be expressed as
\[
v = \vmat{c}{T \\ \hline L }^{-1} \vvec{h \\ \hline \xi}.
\]
Similarly, the vertex $v(\lambda)$ of $P(\lambda)$ is given by
\begin{equation}
v(\lambda) = \vmat{c}{T \\ \hline L }^{-1}
	\left( \vvec{h \\ \hline \xi} - \lambda \vvec{ 1 \\ \hline 0 } \right)
	= v + \lambda q,
	\quad \text{ where }
	q := -\vmat{c}{T \\ \hline L }^{-1} \vvec{ 1.\lambda \\ \hline 0 }.
\label{eq:v-lambda-q}
\end{equation}
Since changing $\lambda$ effectively amounts to changing the right-hand sides
in the formulation of $P(\lambda)$, it does not affect its recession cone.
In other words, letting $C$ be the recession cone of $P$, we have
$P(\lambda) = \{ v(\lambda) \} + C$, for all $\lambda \geq 0$.
Note that since $P(\lambda) \subseteq P$, we know that $q \in C$.

If there exists a ray $r \in C$ of $P$ such that $\alpha^T r < \beta$,
then the set~\eqref{eq:depth-set} is feasible for all values of $\lambda$,
and the optimization problem~\eqref{eq:depth-lp} is unbounded
(i.e., either $P \cap \Z^n$ is empty, or the cut is invalid).
Otherwise,~\eqref{eq:depth-set} is feasible if and only if $v(\lambda)$
is feasible, and~\eqref{eq:depth-lp}
becomes $\max \{ \lambda \geq 0 \, : \, \alpha^T v(\lambda) \leq \beta \}$,
i.e.
\begin{equation}
\max \{ \lambda \geq 0 \; : \; \alpha^T (v + \lambda q) \leq \beta \}.
\label{eq:depth-corner-opt}
\end{equation}
We assume that $\alpha^T v \leq \beta$
(otherwise the cut is either nonviolated, or invalid)
and that $\alpha^T q > \beta$
(otherwise the depth is unbounded again:
either $P \cap \Z^n$ is empty, or the cut is invalid).
We can thus solve~\eqref{eq:depth-corner-opt} in closed form,
finding $\lambda$ such that
$\alpha^T v + \lambda \alpha^T q = \beta$, i.e.
\begin{equation}
\depth_P(\alpha, \beta) = \frac{\beta - \alpha^T v}{\alpha^T q}.
\label{eq:depth-corner-cf}
\end{equation}

\subsection{Computing depths in practice}

For each cut, the cost of computing its depth is that of solving the
linear optimization problem~\eqref{eq:depth-lp} (or, in the case of
corner polyhedra, computing $q$ from the linear system~\eqref{eq:v-lambda-q}).
In standard \emph{inequality} form,
the problem~\eqref{eq:depth-lp} has the same dimension as the original
formulation, plus one row and one column.
Warm start can be exploited
if we are to compute the depth of multiple cuts, since only the
constraint $\alpha^T x \leq \beta$ will change.

Note that in practice, mixed-integer programming solvers work
in standard \emph{equality} form with general upper and lower bounds,
in part because it is what implementations of the simplex method expect.
A formulation would thus be
$\{ x \in P \, : \, x_j \in \Z, \text{ for } j \in J \}$, where
\[
P = \{ x \in \R^n \; : \; L x = \xi, \; \ell \leq x \leq u \}.
\]
Let the constraint matrix $L \in \R^{p \times n}$ be full row rank\footnote{This assumption generally holds true in practice, because presolve attempts to remove redundant rows. Furthermore, some LP solvers assume that the constraint matrix contains an identity, whose columns (possibly fixed to zero by $\ell$ and $u$) are used as slacks, primal phase-I artificial variables, and replacement columns for basis repair.}. The affine hull of $P$ is $\mathcal{L} = \{ x \in \R^n \, : \, L x = \xi \}$ and the $2n$ inequalities are $e_j^T x \geq \ell_j$ and $e_j^T x \leq u_j$ for $j = 1, \ldots, n$.
As we will see below, this makes problem~\eqref{eq:depth-lp} in standard equality form substantially larger than the original formulation.

We now account for the fixed cost of constructing the matrix $T$.
Recall that every row $t^i$ of $T$ is
the normalized projection of $a^i$ on $\mathcal{L}$,
where $a^i$ is a normal vector to the $i$th inequality constraint
(pointing towards infeasibility).
In standard equality form, this means computing the projection of $e_j$ (and $-e_j$) onto the vector subspace $\{ x \in \R^n : Lx = 0 \}$. One approach is adding to $e_j$ a linear combination of the rows of $L$, such that the result is the desired projection. Letting that projection be $e_j + L^T \mu^j$, we need to find $\mu^j \in \R^p$ such that $L (e_j + L^T \mu^j) = 0$. We thus solve the linear system
\[
L L^T \mu^j = -L_j,
\]
where $L_j$ is the $j$th column of $L$. It should be noted that $LL^T$ is a $p \times p$ positive definite matrix, since the rows of $L$ are linearly independent, allowing a Cholesky factorization. The element $(LL^T)_{ik}$ will be zero whenever the $i$th and $k$th rows of the constraint matrix $L$ are orthogonal. Thus, when $L$ is sparse, $LL^T$ should be somewhat sparse as well. From a practical perspective, solving $n$ such systems is a nontrivial computation, but it can be expected to take a fraction of the time necessary to solve the root node LP relaxation\footnote{The simplex method requires solving 2 to 4 unsymmetric linear systems of the same size, per iteration. It is often reasonable to expect about $n$ iterations in practice, although the variance is high and the worst case is, of course, exponential in $n$.}.
Once the $\mu^j$ vectors are computed,
we let ${\gamma^j} := e_j + L^T \mu^j$,
and the bound constraints can be reformulated as
\[
\ell_j + \xi^T \mu^j \leq {\gamma^j}^T x \leq u_j + \xi^T \mu^j.
\]
After normalizing, we obtain
\[
\frac{\ell_j + \xi^T \mu^j}{||{\gamma^j}||}
	\leq \frac{{\gamma^j}^T x}{||{\gamma^j}||}
	\leq \frac{u_j + \xi^T \mu^j}{||{\gamma^j}||},
\]
so the constraints corresponding to $Tx + 1.\lambda \leq h$
in~\eqref{eq:depth-lp} will be
\begin{equation}
\frac{\ell_j + \xi^T \mu^j}{||{\gamma^j}||} + \lambda
	\leq \frac{{\gamma^j}^T x}{||{\gamma^j}||}
	\leq \frac{u_j + \xi^T \mu^j}{||{\gamma^j}||} - \lambda.
\label{eq:std-T}
\end{equation}
Using constraints~\eqref{eq:std-T} in~\eqref{eq:depth-lp} would yield
a standard equality form problem with $2n$ additional linear constraints,
and $2n$ additional slacks, compared to the formulation of $P$.
Furthermore, these new constraints
can be expected to be denser than the original constraints.
This can be partially mitigated by multiplying~\eqref{eq:std-T} by the norm of
the projected vector
\[
\ell_j + \xi^T \mu^j + ||{\gamma^j}|| \lambda
	\leq {\gamma^j}^T x
	\leq u_j + \xi^T \mu^j - ||{\gamma^j}|| \lambda
\]
then subtracting ${\mu^j}^T L x = {\mu^j}^T \xi$
\[
\ell_j + ||{\gamma^j}|| \lambda
	\leq x_j
	\leq u_j - ||{\gamma^j}|| \lambda,
\]
yielding the optimization problem
\begin{equation}
\begin{array}{rll}
\min & \lambda & \\
\st
& \displaystyle x_j - ||{\gamma^j}|| \lambda - y_j = \ell_j
	& \text{for all } j = 1, \ldots, n \\
& \displaystyle x_j + ||{\gamma^j}|| \lambda + z_j = u_j
	& \text{for all } j = 1, \ldots, n \\
& L^i x = \xi_i
	& \text{for all } i = 1, \ldots, p \\
& x \in \R^n, y \in \R^n_+, z \in \R^n_+, \lambda \in \R_+ & \\
\end{array}
\label{eq:std-scaled}
\end{equation}
Like one that would use~\eqref{eq:std-T},
the linear problem~\eqref{eq:std-scaled}
has $p + 2n$ constraints and $3n + 1$ variables.
In~\eqref{eq:std-scaled} however, the added rows are very sparse,
and the original constraints are left untouched,
avoiding numerical errors in the formulation.

\section{Conclusion}
\label{sect:conclusion}

We propose a new measure for the strength of cutting planes,
which we call \emph{depth}.
It is strictly positive for violated inequalities and bounded
above by a function of the dimension and the integer lattice.
We argue (i) that it is a useful theoretical tool, one which can
help us explain computational results observed previously,
(ii) that it should be a good proxy for the computational usefulness
of individual cuts in a branch-and-cut framework,
and as such could be used to help tackle the problem of cut selection,
and (iii) that it can be computed, or at least approximated, at
a reasonable computational cost.
The concept of depth also raises a few questions that we did not address
here.

Regarding full-dimensional integer hull depth,
there is a gap of a factor $\sqrt{2}$ between the upper bound
in Theorem~\ref{thm:depth-ub} and the lower bound in Theorem~\ref{thm:depth-lb}.
In particular, finding the tightest possible bound for
Lemma~\ref{lemma:X} seems like a very simple and elegant
problem, to which we do not have a solution.
We conjecture that Theorem~\ref{thm:depth-lb} is tight
up to an additive constant,
and that Lemma~\ref{lemma:X}/Theorem~\ref{thm:depth-ub} can be strengthened.
Our intuition is motivated by the following:
Let $(\bar x, \bar y)$ be pair of points constructed in~\eqref{eq:depth-ub-cons}
that achieves maximum distance for
the optimiziation problem~\eqref{eq:depth-ub-opt}.
Then, let $X'$ be the convex hull of all the vertices of $X$ except $\bar x$.
It turns out that $\bar y \in X'$.
It is thus possible that there exists a construction for $X$ that
features a smaller ball radius, and still contains $\bar y$.

It would be interesting to have a priori upper bounds on families of cuts
beyond split cuts. We provide a rough bound for intersection cuts in
Section~\ref{sect:ic-bound}, but it can be evaluated only after a cut
is computed. A bound that is parametrized on the lattice-free set would
be more useful both theoretically (to compare different types of lattice-free
sets) and computationally (to select better lattice-free sets).

It is easy to show that the
depth of a cutting plane with respect to a relaxation is an upper bound
on its depth with respect to the original formulation.
We can thus already approximate the depth -- at a much lower computational
cost -- by considering corner relaxations and using~\eqref{eq:depth-corner-cf}.
While the depth of a cut is \emph{not} the minimum of its depths with
respect to all corners,
it may still be possible to get tighter approximations by using multiple bases.
In the case of strict corners (which can have more facets than dimensions),
$P(\lambda)$ may not be a cone for $\lambda > 0$. However, there could
still be an easy way to compute cut depth.

Finally, computational experiments are required to evaluate the adequacy of
depth as comparator for cut selection. Such experiments would require particular
care, since we would need to measure IP solver performance, which
is notoriously sensitive to small perturbations.

\bibliographystyle{plain}
\bibliography{polyhedra}

\end{document}